\documentclass[12pt]{amsart}
\usepackage{amssymb, amsfonts}
\usepackage{hyperref}
\usepackage{verbatim}

\usepackage[retainorgcmds]{IEEEtrantools}

\newcommand{\<}{\langle}
\renewcommand{\>}{\rangle}

\renewcommand{\a}{\alpha}

\renewcommand{\d}{\delta}

\newcommand{\e}{\varepsilon}

\renewcommand{\l}{\lambda}

\newcommand{\cA}{{\mathcal A}}
\newcommand{\cB}{{\mathcal B}}

\newcommand{\cH}{{\mathcal H}}

\newcommand{\bR}{{\mathbb R}}

\newcommand{\bN}{{\mathbb N}}

\newcommand{\lo}{{\mathrm{\ell^1}}}

\newcommand{\lt}{{\mathrm{\ell^2}}}
\newcommand{\lin}{{\mathrm{\ell^\infty}}}
\newcommand{\tB}{{\tilde B}}
\newcommand{\tdA}{{\tilde A}}
\newcommand{\norm}[1]{ \|  #1 \|}
\def\complex{{\mathbb C}\/}
\def\naturals{{\mathbb N}\/}
\def\integers{{\mathbb Z}\/}

\def\Lth{{\mathbb L}}
\def\Clth{C_\Lth}

\newtheorem{thm}{Theorem}[section]
\newtheorem{cor}[thm]{Corollary} 
\newtheorem{prop}[thm]{Proposition} 

\newtheorem{lem}[thm]{Lemma} 

\theoremstyle{definition}   

\newtheorem{defn}[thm]{Definition} 
\newtheorem{exam}[thm]{Example} 
\newtheorem{notation}[thm]{Notation} 

\numberwithin{equation}{section}

\title
[Nilpotent groups]
{Nilpotent group C*-algebras as \\
compact quantum metric spaces  \\
}
\author{Michael Christ and Marc A. Rieffel}
\address{Department of Mathematics \\
University of California \\ Berkeley, CA 94720-3840}
\email{mchrist@berkeley.edu  \\
rieffel@math.berkeley.edu  
}
\date{August 4, 2015}
\thanks{The research reported here was
supported in part by National Science Foundation grants DMS-1066368 and  DMS-1363324. 
}


\begin{document}

\begin{abstract}
Let $\Lth$ be a length function on a group $G$, and let $M_\Lth$ denote the
operator of pointwise multiplication by $\Lth$ on $\lt(G)$. 
Following Connes,
$M_\Lth$ can be used as a ``Dirac'' operator for the reduced
group C*-algebra $C_r^*(G)$.  It defines a
Lipschitz seminorm on $C_r^*(G)$, which defines a metric on the state space of
$C_r^*(G)$. We show that 
for any length function of a strong form of polynomial growth on a discrete group,
the topology from this metric 
coincides with the
weak-$*$ topology (a key property for the 
definition of a ``compact quantum metric 
space''). In particular, this holds for all word-length functions
on finitely generated nilpotent-by-finite groups.
\end{abstract}

\maketitle
\allowdisplaybreaks

\section{Introduction}
The group $C^*$-algebras of discrete groups provide a much-studied class of
``compact non-commutative spaces'' (that is, unital $C^*$-algebras). 
In \cite{Cn1} Connes showed that the ``Dirac'' operator of a spectral
triple over a unital $C^*$-algebra provides in a natural way 
a metric on the state space of the algebra.  The
class of examples most discussed in \cite{Cn1} consists of the  group 
$C^*$-algebras of discrete groups $G$, with the Dirac operator consisting 
of the pointwise multiplication operator on $\ell^2(G)$ 
by a word-length function on the group. 
In \cite{R4, R5}  
the second author pointed out that, motivated by what happens 
for ordinary compact metric spaces, it is natural to desire that 
a spectral triple have the property that the topology from the metric on
the state space coincide with the weak-$*$ topology (for which the 
state space is compact). This property was verified in \cite{R4} 
for certain examples. In \cite{R6} this property was taken as the
key property for the definition of a ``compact quantum metric space''.
This property is crucial for defining 
effective notions of quantum Gromov-Hausdorff distance between
compact quantum metric spaces \cite{R6, R7, R21, R29}.

In \cite{R18} the second author studied this property for Connes'
original class of  examples consisting of discrete groups 
with Dirac operators coming
from a word-length functions, and established that it holds for the group ${\mathbb Z}^n$,
relying on geometric arguments. 
Later, with N.~Ozawa \cite{OzR}, 
he established this property
for hyperbolic groups with word-length functions.
The argument was very different from that in \cite{R18},  
relying on filtered C*-algebras. 

In the present paper we verify the property for the case of
finitely generated nilpotent-by-finite groups 
equipped with length functions of polynomial growth, and 
generalize this to a certain class of length functions on
infinitely generated discrete groups.
Since the approach used in the present
paper is quite different from those used in \cite{R18}  and \cite{OzR},
 this raises the 
question of finding a unified approach which covers both the nilpotent
and hyperbolic settings.
The question of what happens for other classes of groups remains wide open.

 To be more specific, let $G$ be a countable (discrete) 
group, and let
$c_c = C_c(G)$ denote the convolution $*$-algebra of complex-valued 
functions of finite
support on $G$.  Let $\l$ denote the usual $*$-representation of $c_c$ on
$\lt = \lt(G)$ coming from the unitary representation of $G$ by left 
translation on
$\lt$.  
Thus 
\[ \lambda_f(\xi)(x) = f*\xi(x) = \sum_{y\in G} f(xy^{-1})\xi(y)\]
for functions $\xi\in\ell^2(G)$.
The completion of $\l(c_c)$ for the operator norm is by definition 
the reduced
group $C^*$-algebra, $C_r^*(G)$, of $G$.  We identify $c_c$ with 
its image in
$C_r^*(G)$, so that it is a dense $*$-subalgebra. We remark that by sending
an element $a \in C^*_r(G)$ to the element of $\lt$ to which it sends
$\d_e \in \lt$ we obtain an embedding of $C^*_r(G)$ into $\lt$. Thus 
when convenient
we can view all of the elements of $C^*_r(G)$ as functions on $G$.
We denote by $e$ the identity element of $G$.

The F\o lner condition for amenability \cite{Mnn, Ptn} 
is a simple consequence of polynomial growth (in the weakest of the three
versions defined below).
Consequently the full and reduced group C*-algebras coincide \cite{Ptn}
under our hypotheses, and so
we do not need to distinguish between them.

Let a length function $\Lth$ be given on $G$.  That is, $\Lth$
is a function from $G$ to $[0,\infty)$ that satisfies
\begin{enumerate}
\item $\Lth(xy) \leq \Lth(x) + \Lth(y)$ for all $x, y \in G$;
\item $\Lth(x^{-1}) = \Lth(x)$ for all $x \in G$;
\item $\Lth(x) = 0$ if and only if $x = e$. 
\end{enumerate}
We say that $\Lth$ is proper if 
$B(r)=\{x\in G: \Lth(x)\le r\}$ is a finite subset of $G$
for each $r<\infty$.

Throughout the paper, we denote by $|E|$ the cardinality of a finite set $E$.

In the literature there are actually two (or more) inequivalent 
definitions of ``polynomial growth". Since we want to distinguish between
them, we will call one of them ``strong polynomial growth''. The proof of
our main theorem works most naturally for an intermediate property,
which we call ``bounded doubling''. 

\begin{defn}
\label{polygr}
Let $\Lth$ be a length function on a group $G$. We say that $\Lth$ has (or is of)
\begin{enumerate}
\item
\emph{strong polynomial growth} if $\Lth$ is proper and
there exist constants $\Clth<\infty$ and $d < \infty$
such that 
\begin{equation}
\label{polyg} 
\Clth^{-1}r^d \leq |B(r)| \leq \Clth r^d   \ \text{ for all $r\ge 1$.}  \end{equation}
\item
\emph{bounded doubling} if $\Lth$ is proper and
there exists a constant $\Clth<\infty$
such that 
\begin{equation}
\label{polygrowthdefn} 
|B(2r)| \leq \Clth |B(r)|   \ \text{ for all $r\ge 1$.}  \end{equation}
\item
\emph{polynomial growth} if $\Lth$ is proper and
there exist constants $\Clth<\infty$ and $d < \infty$
such that 
\begin{equation}
\label{wpolyg} 
|B(r)| \leq \Clth r^d   \ \text{ for all $r\ge 1$.}  \end{equation}
\end{enumerate}
\end{defn}

Equivalent definitions are obtained by changing
the restriction $r\ge 1$ to $r\ge r_0$ for any $r_0>0$,
but the constants $\Clth$ may depend on $r_0$.

\begin{prop}
\label{bnded}
Let $\Lth$ be a length function on a group $G$. If $\Lth$
has strong polynomial growth, then it has bounded doubling. 
If $\Lth$ has bounded doubling then it has polynomial
growth. If $G$ is finitely generated, then these three
properties are equivalent. But in general, no two of these properties are equivalent.
\end{prop}

See Section~\ref{examp} for a proof, and for examples illustrating these distinctions.

We let $M_h$ denote the (often
unbounded) operator on $\lt$ of pointwise multiplication by 
a function $h:G\to\complex$.  The multiplication operator
$M_\Lth$ will serve as our ``Dirac'' operator, and we will
denote it by $D$.  One sees easily 
\cite{Cn7, R18, OzR} 
that the commutators $[D,\l_f]$ are bounded operators for each $f \in c_c$.  We
can thus define a seminorm, $L_D$, on $c_c$ by $L_D(f) =
\|[D,\l_f]\|$,
where $\|T\|$ denotes the operator norm of a bounded linear operator
$T:\ell^2(G)\to\ell^2(G)$. (Connes points out in proposition 6 of
\cite{Cn7} that $\Lth$ has polynomial growth exactly if there is a 
positive constant, $p$,
such that the operator $D = M_\Lth$ is such that 
$(1 + D^2)^{-p}$ is a trace-class operator.)

Let $L$ be a ${*}$-seminorm (i.e. $L(a^*) = L(a)$)
on a dense $*$-subalgebra $A$ of a unital
$C^*$-algebra ${\bar A}$, satisfying $L(1) = 0$. Define a metric, $\rho_L$, on
the state space $S({\bar A})$ of ${\bar A}$, much as Connes did, by
\[ \rho_L(\mu,\nu) = \sup\{|\mu(a) - \nu(a)|: a \in A,\ L(a) \le 1\}.  \]
(Without further hypotheses, $\rho_L$ may take the value $+\infty$.)  

\begin{defn} \cite{R5}
A ${*}$-seminorm $L$ on $A$ is a 
{\em Lip-norm} if the topology on $S({\bar A})$ defined by
the associated metric $\rho_L$ coincides with the weak-$*$ topology.  
\end{defn}

We consider a unital $C^*$-algebra equipped with a
Lip-norm $L$ to be a compact quantum metric space \cite{R6}, but for
many purposes one wants $L$ to satisfy further properties.
See the discussion after Proposition \ref{prosaic}.
The main question that we deal with in this paper is whether the seminorms $L_D$
defined as above in terms of length functions $\Lth$ on discrete groups are Lip-norms.
Our main theorem is: 

\begin{thm}
\label{mainth}
Let $G$ be a discrete group, and let $\Lth:G\to[0,\infty)$
be a length function of bounded doubling on $G$.  
Let $D=M_\Lth$ be the associated multiplication operator.
Then the seminorm $L_D$ defined on
$c_c$ by $L_D(f) = \|[D,\l_f]\|$ is a Lip-norm on $C^*(G)$.
\end{thm}

Necessary and sufficient conditions for a seminorm on a
pre-$C^*$-algebra
to be a Lip-norm are given in \cite{R4, R5} (in a more general context).  For our
present purposes it is convenient to reformulate these conditions slightly.
The following reformulation is an immediate corollary of proposition~1.3
of \cite{OzR}.

\begin{prop}
\label{proplip}
Let $G$ be a discrete group, and let $\Lth:G\to[0,\infty)$ be a length function. 
The associated seminorm $L_D$ is a Lip-norm on $c_c=C_c(G)$
if and only if $\l$ carries
\[ \{f \in c_c:   f(e) = 0    \mbox{ and }  L_D(f) \le 1 \} \]
to a subset of $\cB(\lt)$ that is totally bounded for the operator norm.
\end{prop}

Accordingly,  the content of this paper consists in verifying
the criterion of this proposition for the case of a
group $G$ equipped with a length function $\Lth$ that has bounded doubling.

Shorn of its functional analytic context and motivation,
the result proved in this paper is as follows.
The proof developed below is loosely related to some elements of
\cite{Crs1} and \cite{Crs2}.

\begin{prop} \label{prosaic}
Let $G$ be a discrete group. Let $\Lth:G\to[0,\infty)$ be a 
length function on $G$ that has bounded doubling,
and let $D_\Lth$ be the associated Dirac operator on $c_c(G)$.
For every $\varepsilon>0$ there exists a finite set $S_\varepsilon\subset G$
such that for any finitely supported $f:G\to\complex$
satisfying $\norm{[D_\Lth,\lambda_f]}\le 1$
there exists a decomposition $f = f_\sharp+f_\flat$
such that $f_\flat$ is supported on $S_\varepsilon$ and 
$\norm{\lambda_{f_\sharp}}\le\varepsilon$.
\end{prop}

More generally, for an arbitrary function $f:G\to\complex$,
$[D_\Lth,\lambda_f]$ is well-defined as a linear operator from $c_c$ to the space of
all functions from $G$ to $\complex$.
The analysis below demonstrates that if $f:G\to\complex$
is any function for which $[D_\Lth,\lambda_f]$
maps $c_c$ to $\ell^2$ and extends to a bounded linear operator
from $\ell^2$ to $\ell^2$ with $\norm{[D_\Lth,\lambda_f]}\le 1$, then $f$ satisfies the 
conclusion of Proposition~\ref{prosaic}. 
In particular, 
$f$ (that is, $\lambda_f$) is necessarily an element of $C^*_r(G)$.

We believe that our whole discussion could be extended to
the slightly more general setting of group $C^*$-algebras 
twisted by a $2$-cocycle \cite{pck1, pck2}, 
much as done in \cite{R18}, but we have not checked this carefully.

The definition of a
``compact C*-metric'' as given in definition 4.1 of \cite{R21}
brings together
most of the additional conditions that have been found to be 
useful to require of
a Lip-norm $L$ on a C*-normed algebra $\cA$.
Namely, one wants $L$ to be lower semi-continuous with respect
to the operator norm, to be strongly Leibniz as defined there, and
one wants the $*$-subalgebra of elements of $\cA$ 
on which $L$ is finite to be a 
dense spectrally-stable subalgebra of the norm-completion $\bar \cA$
of $\cA$. For any group $G$ with proper length function $\Lth$
and corresponding seminorm $L_D$ for $D = M_\Lth$ 
one can always obtain these properties in the following way (as 
explained in \cite{R21}, especially its example 4.4).
The one-parameter unitary group generated by $D$ consists
of the operators of pointwise multiplication by the functions
$e^{it\Lth}$. Conjugation by these operators defines a 
one-parameter group, $\a$, of automorphisms of
$\cB(\lt)$ (which need not be strongly continuous, and need
not carry $\cA = C^*_r(G)$ into itself). By using $\a$ one shows that $L_D$
on $c_c$ is lower semi-continuous with respect to the operator
norm, and so has a natural extension, $\bar L_D$ to a lower
semi-continuous seminorm on all of $\bar \cA = C^*_r(G)$ 
(which may take the
value $+ \infty$). Let $\cA^\infty$ denote the $*$-subalgebra of 
elements of $\bar\cA$ that are infinitely differentiable for $\a$
It contains $c_c$ and so is dense in $\bar \cA$, and it is
spectrally stable in $\bar \cA$. The restriction of $\bar L_D$
to $\cA^\infty$ satisfies all the conditions for being a $C^*$-metric,
for reasons given in section 3 of \cite{R21}, 
except for the fact that it may not be a Lip-norm.
Thus this paper verifies, for groups with length
functions of bounded doubling, 
the most difficult condition, namely of obtaining
a Lip-norm, so that for such groups
$(\cA^\infty, \bar L_D)$ is a compact C*-metric space. 
One can continue to show that all continues to work well
for matrix algebras over $\cA$ along the lines given in \cite{R29},
so that one should give the definition of a ``matricial $C^*$-metric'',
but we will not pursue that important aspect here.

Since both nilpotent-by-finite groups and hyperbolic groups are
groups of ``rapid decrease'' \cite{Jo2, dHp}, 
it is natural to ask whether our main
theorem extends to all groups of rapid decrease. For the
reader's convenience we recall here the definition of this concept:
For any group $G$ and length function $\Lth$ 
on it, and for any $s \in \bR$, the Sobolev space $\cH^s_\Lth(G)$
is defined to be the set of functions $\xi$ on $G$ such that
$(1+\Lth)^s\xi\in\lt$. The space $\cH^\infty_\Lth$ of rapidly decreasing functions
is defined to be $\bigcap_{s \in \bR}\cH^s_\Lth$.
The group $G$ is said to be of rapid decrease if 
it has a length function $\Lth$ such that $\cH^\infty_\Lth$ is
contained in $C^*_r(G)$, that is, if all the convolutions of elements
of $c_c$ by elements of $\cH^\infty_\Lth$ extend to  bounded operators on $\lt$.
For closely related Lip-norms (which are not Leibniz)
obtained by using ``higher derivatives'' for
groups of rapid decrease, see \cite{AnC}.


\section{Localized weighted inequality}
\label{ineq}

In this section we develop a key inequality that holds for any discrete group $G$
equipped with a proper length function $\Lth$. For any $h \in \lin$ we let $M_h$ 
denote the operator on $\lt$ of pointwise
multiplication by $h$. If $E$ is a subset of $G$, we let $M_E$ 
denote $M_h$ for $h$ the characteristic
(or indicator) function $\chi_E$ of $E$, so $M_E=M_{\chi_E}$ is a projection operator. 
For any $r\geq 0$ we set $B(r) = \{x\in G: \Lth(x) \leq r\}$,
which is a finite set since $\Lth$ is proper.
We set $M_r = M_{B(r)}$.  Each $M_r$ is a spectral projection of $D$.

It is convenient to use the kernel functions for the operators $\l_f$ and $[D, \l_f]$,
for any $f \in c_c$. The kernel function for $\l_f$ is  $f(xy^{-1})$, that
is, $(\l_f\xi)(x) = \sum_y f(xy^{-1})\xi(y)$ for any $\xi \in \lt$. 
The kernel function $[D,\l_f](x,y)$ for the operator $[D, \l_f]$ 
is $[D,  \l_f](x,y) = (\Lth(x)-\Lth(y))f(xy^{-1})$,
with slight abuse of notation. Thus if $\Lth(x) \neq \Lth(y)$ then 
\[ 
f(xy^{-1}) = (\Lth(x) - \Lth(y))^{-1}[D, \l_f](x,y). 
 \]
If $\Lth(x) > \Lth(y)$ then
\[
(\Lth(x) - \Lth(y))^{-1} = \Lth(x)^{-1}(1 - \Lth(y)/\Lth(x)) 
= \Lth(x)^{-1} \sum_{k=0}^\infty \Lth(y)^k \Lth(x)^{-k}   .
\]
Thus, if we are given $r, s \in [0,\infty)$ with $0 \le r < s$, and if $\xi \in \lt$ is supported
in $B(r)$, then for any $x \in G$ satisfying $\Lth(x) \ge s$ we have
\begin{align*}
(\l_f\xi)(x) &= \sum_y f(xy^{-1})\xi(y) = \sum_y (\Lth(x) - \Lth(y))^{-1}[D, \l_f](x,y)\xi(y)  \\
&=\sum_{y \in B(s)} \Lth(x)^{-1} \sum_k \Lth(x)^{-k} \Lth(y)^k [D, \l_f](x,y) \xi(y)   \\
&=\sum_k \Lth(x)^{-1} \sum_{y \in B(s)} \Lth(x)^{-k} [D, \l_f](x,y)  \Lth(y)^k\xi(y)  \\
&=\big(\sum_k D^{-1-k}(I-M_s)[D, \l_f]  D^kM_r\xi\big) \ (x).
\end{align*}
That is,
\[
(I-M_s) \l_f M_r = \sum_{k=0}^\infty D^{-1-k}(I-M_s)[D, \l_f]  D^kM_r  .
\]
But $\|D^{-1-k}(I-M_s)\| \le s^{-1-k}$ while $\|[D, \l_f]  D^kM_r\|  \leq r^kL_D(f)$.
Consequently
\[
\|(I-M_s) \l_f M_r\| \  \le \ \ s^{-1} \sum_k (r/s)^kL_D(f) = (s-r)^{-1}L_D(f)  .
\]
We have thus obtained:
\begin{prop}
\label{weight}
For any $f \in c_c$ and any $r, s \in \bR$ with $s > r \geq 0$ we have
\[
\|(I-M_s) \l_f M_r\| \  \leq \  (s-r)^{-1}L_D(f)  . 
\]
\end{prop}

Let us compare this proposition with the main result of section 2 of \cite{OzR}.
Suppose that $\Lth$ takes its values in $\bN$, and for each $n \in \bN$ let $\cA_n$
consist of the elements of $c_c$ supported on $B(n)$. 
Let $\cA$ denote the union of the $\cA_n$'s, so that $\cA$ is a unital dense
$*$-subalgebra of $\lo(G)$. Then the family $\{\cA_n\}$ 
is a filtration of $\cA$, and in the topological sense it is a filtration of
$\lo(G)$, and of the C*-algebra completion $C^*_r(G)$ 
of $\lo(G)$ for the operator norm. This is discussed in section 1 of \cite{OzR},
where the following observations are made. For a faithful tracial state on a
filtered C*-algebra  (such as the canonical trace on $C^*_r(G)$) with filtration
$\{\cA_n\}$, one can form
the corresponding GNS Hilbert space, $\cH$, and the representation $\l$
of $\cA$ on it coming from the left regular representation of $\cA$ on itself. 
For each $n \in \bN$ let 
$Q_n$ denote the orthogonal
projection of $\cH$ onto its (finite-dimensional) 
subspace  $\cA_n$. (In the above discussion for groups this
operator would be denoted by $M_n$.) Then set $P_n = Q_n - Q_{n-1}$ 
for $n \geq 1$, and $P_0 = Q_0$. The $P_n$'s are mutually orthogonal,
and their sum is $I_\cH$ for the strong operator topology. One then
defines an unbounded operator $D$ on $\cH$ by $D = \sum_{n = 0}^\infty nP_n$.
For any $a \in \cA$ the densely defined operator $[D, \l_a]$ is a bounded
operator, and so extends to a bounded operator on $\cH$. We can then define
a seminorm, $L_D$, on $\cA$ by
\[
L_D(a) = \|[D, \l_a]\|.
\]
This $L_D$ is essentially a generalization of the $L_D$ that we have used
above for the group case.
Let $T$ be any bounded operator on $\cH$ such that $[D, T]$ has dense
domain containing $\cA$  and is bounded on its domain, so extends
to a bounded operator on $\cH$. For any natural number $N$ set
\[
T^{(N)} = \sum_{|m-n| > N} P_mTP_n.
\]
Then the main result of section 2 of \cite{OzR} 
provides a specific sequence, $\{C_N\}$, of constants, independent of $D$ and $T$, 
that converges to 0 as $N$ goes
to $\infty$, such that
\[
\|T^{(N)}\| \leq C_N\|[D, T]\|
\]
for all $N$. Notice then that for any $p,  q \in \bN$ such that $q -p >N$
we have
\[
(1-Q_q)\big(\sum_{|m-n| > N} P_mTP_n \ \big)Q_p = (1-Q_q)TQ_p,
\]
and consequently
\begin{equation} \label{QpQq}
\|(1-Q_q)TQ_p)\| \leq C_N\|[D,T]\|   .
\end{equation}
This is essentially a generalization of Proposition~\ref{weight}, but with not as good
a constant. 


\section{Cutoff functions}
\label{furt}

For the proof of Theorem \ref{mainth} we seek, for any $\e > 0$ and every
$f \in c_c$, 
a decomposition $f=f_\sharp+f_\flat$ with certain properties.
It is natural to accomplish this by means of multiplication operators,
so that in the notation of Proposition~\ref{prosaic}, 
$f_\flat= M_g f = g f$ where the cutoff function $g$ depends only on $G$, $\Lth$, and $\e$.
It will be more convenient to construct $f_\sharp$ in this way,
and this will be accomplished by means of an infinite series of finitely supported cutoff functions.
Thus one is led to analyze $\lambda_{g_\nu f}$ in terms of $\lambda_f$, 
for a family of cutoff functions $g_\nu$ whose supports are finite for each $\nu$, but not uniformly so.

As motivation, consider the Abelian case, employing additive notation $x-y$ in place of
multiplicative $xy^{-1}$ for the group operation.
The operator $\lambda_{gf}$ has kernel function $g(x-y)f(x-y)$.
As in the proof of Proposition~\ref{weight},
it can be useful to express $g$ as an infinite sum of product functions
$g(x-y) = \sum_k \phi_k(x)\psi_k(y)$ 
with $\sum_k \norm{\phi_k}_{L^\infty} \norm{\psi_k}_{L^\infty}\le C_0$, 
where $C_0$ is a finite constant which is to be bounded uniformly over a suitable family 
of cutoff functions $g$. 
This expresses $\lambda_{gf}$ as $\sum_k M_{\phi_k} \lambda_f  M_{\psi_k}$
 with $\sum_k \norm{M_{\phi_k}\lambda_f M_{\psi_k}} \le C_0\norm{\lambda_f}$.
If the Fourier transform $\widehat{g}$ satisfies $\norm{\widehat{g}}_{L^1}\le C_0$ 
then one obtains at once a continuum
decomposition of this type; 
\[
g(x-y) = \int \widehat{g}(\xi)e^{2\pi i \xi\cdot(x-y)}\,d\xi 
= \int \widehat{g}(\xi) e^{2\pi i \xi\cdot x} e^{-2\pi i \xi\cdot y}\,d\xi   ,
\]    
and one sets $\phi_\xi(x)=\widehat{g}(\xi)e^{2\pi i x\cdot\xi}$
and $\psi_\xi(y) = e^{-2\pi i y\cdot\xi}$
to obtain 
\[
\int \norm{\phi_\xi}_\infty \norm{\psi_\xi}_\infty\,d\xi\le C_0.
\]
One effective way to ensure that $\norm{\widehat{g}}_{L^1}\le C_0$ 
is to express $g$ as a convolution product $g=g_1*g_2$ with
$\norm{g_1}_{\ell^2}\norm{g_2}_{\ell^2}\le C_0$.
For not necessarily Abelian groups with length functions of bounded doubling,
we will show below how convolution products of appropriately chosen $\ell^2$ functions
can be used to construct useful
cutoff functions $g$, despite the lack of a convenient Fourier transform.

\subsection{Convolutions as cutoff functions}

We begin with some generalities concerning $\lambda_{gf}$ when the cutoff function $g$
is expressed as a convolution 
$h^**k$ for $h, k \in c_c$. 
Let $\rho$ denote the right regular representation
of $G$ on $\lt$, defined by $\rho_u(\xi)(x) = \xi(xu^{-1})$. Then
$\rho_u$ commutes with $\l_f$ for any $f \in c_c$. 
For any $h\in c_c$ we define $\tilde h(x) = h(x^{-1})$
and $h^*(x) = \overline{h}(x^{-1})$.

\begin{prop}
\label{cutoff} 
For any $f, h, k \in c_c$ we have
\begin{equation} \label{series*}
\l_{(h^**k)f} = \sum_z \rho_z^* M^*_{\tilde h} \l_f M_{\tilde k} \rho_z,
\end{equation}
where this sum converges for the weak operator topology. 
Furthermore 
\[ \|\l_{(h^**k)f}\| \leq \|\l_f\| \|h\|_2 \|k\|_2.  \]
\end{prop}

\begin{proof}
Notice that
\[
(h^**k)(yx^{-1}) = \sum_z \bar h(z^{-1})k(z^{-1}yx^{-1}) 
= \sum_z \bar h(z^{-1}y^{-1})k(z^{-1}x^{-1}).
\]
Then, on using this, for any $\xi, \eta \in c_c$ we have
\begin{align*}
 \<\l_{(h^**k)f} \xi, \eta\> & = \sum_y(\l_{(h^**k)f} \xi)(y)\bar \eta(y)  \\
&= \sum_y \sum_x (h^**k)(yx^{-1})f(yx^{-1})\xi(x)\bar\eta(y)   \\
&= \sum_y \sum_x \sum_z \bar h(z^{-1}y^{-1})k(z^{-1}x^{-1})f(yx^{-1}))\xi(x)\bar\eta(y)  \\
&= \sum_y \sum_x \sum_z  f(yx^{-1})k(x^{-1})\xi(xz^{-1})\bar h(y^{-1})\bar \eta(yz^{-1})   \\
&= \sum_z \<\l_f M_{\tilde k} \rho_z \xi, \ M_{\tilde h} \rho_z \eta\>  \\
& =  \sum_z \<\rho_z^*M_{\tilde h}^*\l_f M_{\tilde k} \rho_z \xi, \   \eta\>       .
\end{align*}
But
\begin{align*}
\sum_u \| M_{\tilde k}\rho_u \xi\|^2_2 &= \sum_u\sum_x|M_{\tilde k}\rho_u \xi (x)|^2
=\sum_u\sum_x|k(x^{-1})\xi(xu^{-1})|^2  \\
&= \sum_x |k(x^{-1})|^2 \|\xi\|^2_2 = \|k\|^2_2 \|\xi\|^2_2   ,
\end{align*}
and similarly for $M_{\tilde h}\rho_v\eta$, so that by Cauchy-Schwarz,
\[
\sum_z |\<\l_f M_{\tilde k} \rho_z \xi, \ M_{\tilde h} \rho_z \eta\> |  
\leq \|\l_f\| \|h\|_2 \|k\|_2 \|\xi\|_2 \|\eta\|_2.
\]
This implies both convergence of the series \eqref{series*} for the weak operator topology, and
the stated norm inequality.
Notice that because $\rho$ is a unitary representation the norm of
each operator $\rho_z^*M_{\tilde h}^*\l_f M_{\tilde k} \rho_z $ is equal to
$\|M_{\tilde h}^*\l_f M_{\tilde k}\|$.
\end{proof}

Proposition~\ref{cutoff} fits very well into the setting of ``proper actions of
groups on C*-algebras'' that is defined and discussed in \cite{R30}. Let
$\cA$ denote the algebra of compact operators on $\lt$, and let $\a$
denote the action of $G$ on $\cA$ by conjugation by $\rho$. From
example~2.1 of \cite{R30} but with the roles of $\l$ and $\rho$ reversed,
we see that $\a$ is a proper action as defined in \cite{R30}. The 
finite-rank operator $M_{\tilde h}^*\l_f M_{\tilde k}$ above is easily
seen to have kernel function of finite support, putting it in the dense
subalgebra $\cA_0$ of example~2.1 of \cite{R30}. Accordingly
$\sum_z \a_z(M_{\tilde h}^*\l_f M_{\tilde k})$ exists in the weak sense
discussed in \cite{R30}, and this sum is an element of the
``generalized fixed-point algebra'' for $\a$ as defined in \cite{R30}.
Towards the end of example~2.1 it is explained that this
generalized fixed-point algebra is, in the case of this example,
just the C*-algebra generated by the left regular representation
(for the roles reversed). Our proposition above yields $\l_{(h^**k)f}$,
which is indeed in this C*-algebra. This general setting is 
explored further in \cite{R31}, especially in sections 7 and 8. 

We do not, strictly speaking, need the following proposition, but
it provides some perspective on the path that we will take below,
e.g. in Proposition~\ref{jip2}. 

\begin{prop}
\label{lip}
Let $f, h, k \in c_c$. Then
\[ L_D(\l_{(h^**k)f}) \leq \|h\|_2\|k\|_2L_D(f).  \]
\end{prop}

\begin{proof}
Because $(h^**k)f $ has finite support, $[D, \l_{(h^**k)f} ]$ is
a bounded operator. Let $\xi, \eta \in c_c$, so they are in the domain of $D$. Then
\[
\<[D,\l_{(h^**k)f} ] \xi, \ \eta\> 
= \<\l_{(h^**k)f} \xi, \ D\eta\> - \<\l_{(h^**k)f} D\xi, \ \eta\>  ,
\]
so by Proposition \ref{cutoff} 
\[
\<[D,\l_{(h^**k)f} ] \xi, \ \eta\> 
= \sum_z \<D\rho_z^*M_h^* \l_f M_k\rho_z \xi, \ \eta\> - 
\<\rho_z^*M_h^*\l_fM_k\rho_z D\xi, \ \eta\>.
\]
But, if by slight abuse of notation we let $\rho_z(h)$ denote the corresponding
right translate
of $h$, we see that $\rho_z^*M_h^*\rho_z = M^*_{\rho_z(h)}$, which commutes with $D$,
and similarly for $M_k$. Furthermore $\rho_z $ commutes with $\l_f$. It follows that 
\[
\<[D,\l_{(h^**k)f} ] \xi, \ \eta\> 
=\sum_z \<\rho_z^*M_h^*[D, \l_f]M_k\rho_z \xi, \ \eta \>.
\] 
Consequently
\[
|\<[D, \l_{(h^**k)f} ]\xi, \ \eta\>| \leq L_D(f) \|h\|_2\|k\|_2\|\xi\|_2\|\eta\|_2
\]
for much the same reasons as given near the end of the proof of
Proposition~\ref{cutoff}.
\end{proof}


\subsection{The seminorm $J_D$ 
and cutoff functions}

Later in the proof we will partly lose control 
of $L_D(gf)$ for certain functions $g$ of interest. 
It is possible to retain some control, as follows. 
Notice that if, for any $r > 0$, we set $s = 2r$
in Proposition \ref{weight}, we obtain
\[
\|(I-M_{2r}) \l_f M_r\| \  \le \  r^{-1}L_D(f)  . 
\]
This motivates the following definition.

\begin{defn}
\label{control}
The seminorm $J_D$ on $c_c$ is defined by 
\[ J_D(f) = \sup\{r\|(I-M_{2r}) \l_f M_r\|:   r>0\}  \]
for any $f \in c_c$. 
\end{defn}
The inequality
\[
J_D(f) \leq L_D(f) \ \text{ for all $f \in c_c$}
\]
is an equivalent formulation of the special case $s=2r$
of Proposition~\ref{weight}. 

We emphasize that for the rest of this section, and for much of the next, 
we use $J_D$ but not $L_D$, although some steps do
have versions for $L_D$. Only near the end
of the next section will we use the fact that  $J_D \leq L_D$.
We will need:

\begin{prop}
\label{norm}
Let $f \in c_c$. If $f(x) \neq 0$ for some $x \neq e$, then $J_D(f) \neq 0$.
Thus the seminorm $J_D$ is a norm on the subspace
$\{f \in c_c: f(e) = 0\}$.
\end{prop}

\begin{proof}
Let $\d_e$ be the delta-function at $e$, viewed as an element of $\lt$.
Then for any $r > 0$ we have $((I-M_{2r}) \l_f M_r)(\d_e) = (I-M_{2r})(f)$,
where on the right-hand side $f$ is viewed as an element of $\lt$.
Let $x \in G$ be such that $f(x) \neq 0$ and $x \neq e$ so that
$\Lth(x) \neq 0$. Choose $r>0$ such that $2r < \Lth(x)$. Then
$(M_{2r}f)(x) = 0$, so that $(I-M_{2r})(f)(x) \neq 0$, and thus
$J_D(f) \neq 0$.
\end{proof}

We now proceed to develop
properties of $J_D$ with respect to cutoffs of functions.

\begin{prop}
\label{bound}
For a given $r>0$, suppose that $h$ is supported on
$G \setminus B(2r)$ and that $k$ is supported on $B(r)$. 
Then for any $f \in c_c$ we have
\[
\|\l_{(h^*k)f}\| \le r^{-1} \|h\|_2 \|k\|_2 J_D(f)   .
\]
\end{prop}
\begin{proof}
For any $\xi, \eta \in c_c$ we have, by Proposition \ref{cutoff}, 
\begin{align*}
| \<\l_{(h^**k)f} \xi, \ \eta\> | &= |\sum_z  \<\l_f M_{\tilde k}\rho_z \xi, \ M_{\tilde h}\rho_z\eta\> | \\
&=  |\sum_z  \<(I-M_{2r})\l_f M_r M_{\tilde k}\rho_z \xi, \ M_{\tilde h}\rho_z\eta\> |   \\
& \le r^{-1}J_D(f)\sum_z \| M_{\tilde k}\rho_z \xi\|_2\|M_{\tilde h}\rho_z\eta\|_2    \\
&\leq r^{-1}(\sum_u \| M_{\tilde k}\rho_u \xi\|^2_2)^{1/2}    
(\sum_v\|M_{\tilde h}\rho_v\eta\|^2_2)^{1/2} J_D(f)    \\
& = r^{-1} \|h\|_2 \|k\|_2    \|\xi\|_2 \|\eta\|_2 J_D(f) ,
\end{align*}
for reasons given near the end of the proof of Proposition \ref{cutoff}.
\end{proof}

Quite parallel to Proposition \ref{lip} we have:

\begin{prop}
\label{jip1}
Let $f, h, k \in c_c$. Then
\[
J_D((h^**k)f) \leq \|h\|_2\|k\|_2 J_D(f)    .
\]
\end{prop}

\begin{proof}
The justifications for the calculations in the proof are very similar to those in 
the proof of Proposition \ref{lip}.
For any $r>0$ we have, by Proposition \ref{cutoff}, 
\begin{align*}
&| \<(I-M_{2r})\l_{(h^**k)f}M_r \xi, \ \eta\> | = | \<\l_{(h^**k)f}M_r \xi, \ (I-M_{2r})\eta\> |   \\
&=  |\sum_z  \<\l_f  M_{\tilde k}\rho_z M_r \xi, \ M_{\tilde h}\rho_z(I-M_{2r})\eta\> |   \\
&=  |\sum_z  \<(I-M_{2r})\l_f M_r  M_{\tilde k}\rho_z  \xi, \ M_{\tilde h}\rho_z(I-M_{2r})\eta\> |   \\
&\leq  \sum_z  |\<(I-M_{2r})\l_f M_r  M_{\tilde k}\rho_z  \xi, \ M_{\tilde h}\rho_z(I-M_{2r})\eta\> |   \\
&\leq r^{-1}  \|h\|_2\|k\|_2 J_D(f) \|\xi\|_2\|\eta\|_2    ,
\end{align*}
for reasons given near the end of the proof of Proposition \ref{cutoff}.
\end{proof}

\begin{cor}
\label{subsets}
For given $r>0$, suppose that $E \subset B(r)$ and $F \subset G\setminus B(2r)$,
and set $k = \chi_E$ and $h = \chi_F$. Then for any $f \in c_c$ we have
\[
\|\l_{(h^**k)f}\| \leq r^{-1} |E|^{1/2} |F|^{1/2} J_D(f)   ,
\]
and
\[
J_D((h^**k)f) \leq  |E|^{1/2} |F|^{1/2} J_D(f)   .
\]
\end{cor}

\subsection{Cutoff functions approximating indicator functions of annuli}

\begin{notation}
\label{annulus}
For $t > s > 0$ we 
define the annulus $A(s,t)$ to be
\begin{equation} 
A(s,t) = B(t)\setminus B(s) =  \{x\in G: s<\Lth(x)\le t\}. 
\end{equation}
\end{notation}

\begin{cor}
\label{power}
For given $t >s>2r>0$ let $k = |B(r)|^{-1}\chi_{B(r)}$ and
$h = \chi_{A(s,t)}$, and let $g = h^**k$. Then for any $f \in c_c$ we have
\[
\|\l_{gf}\| \leq r^{-1}(|B(r)|^{-1}|B(t)|)^{1/2} J_D(f)   ,
\]
and
\[
J_D(gf) \leq (|B(r)|^{-1}|B(t)|)^{1/2}J_D(f)    .
\]
\end{cor}
One can consider here that we are interested in restricting $f$ to $A(s,t)$, as $\chi_{A(s,t)}f$, 
but we are first ``smoothing'' $\chi_{A(s,t)}$ by convolving it with the probability
function $k$ centered at 0, to give $gf$.

The following facts are easily verified:

\begin{lem}
\label{support}
For $g$ defined as in Corollary \ref{power}, 
we have $0 \leq g \leq 1$, and furthermore
\begin{itemize}
\item[a)] If $g(x) \neq 0$ then $s-r < \Lth(x) \leq t+r$, that is, $x \in A(s-r,t+r)$.
\item[b)] If $x \in A(s+r,t-r)$, that is, $s+r < \Lth(x) \leq t-r$, then $g(x) = 1$.
\end{itemize}
\end{lem} 

For later use we draw the following consequences from 
Corollary \ref{power} and the above lemma. 
Suppose that $t > s >2r >0$, 
and suppose that $f \in c_c$ vanishes identically
on both the annuli $A(s-r, s+ r)$ and $A(t-r, t+r)$ . Then
\[
\|\l_{f\chi_{A(s+r, t-r)}}\| \leq r^{-1}(|B(r)|^{-1}|B(t)|)^{1/2} J_D(f)   .
\]
and
\[
J_D(f\chi_{A(s+r, t-r)}) \leq (|B(r)|^{-1}|B(t)|)^{1/2}J_D(f)    .
\]
If we reparametrize this inequality by sending $t$ to $t+r$ and $s$ to $s-r$
we obtain the following result:

\begin{prop}
\label{keyprop}
Suppose that $t > s >3r >0$, 
and suppose that $f \in c_c$ vanishes identically on both
the annuli $A(s-2r,s )$ and $A(t, t+2r)$. Then
\[
\|\l_{f\chi_{A(s, t)}}\| \leq r^{-1}(|B(r)|^{-1}|B(t+r)|)^{1/2} J_D(f).
\]
and
\[
J_D(f\chi_{A(s, t)}) \leq (|B(r)|^{-1}|B(t + r)|)^{1/2}J_D(f).
\]
\end{prop}


\section{Application to Nilpotent-by-finite groups}
\label{nilp}

We assume for the remainder of the paper that $\Lth$ is a length function
with the property of bounded doubling. 

\begin{notation}
\label{noter}
For a fixed $R \in \bR$ with $R \geq 2$,  and for any 
natural numbers $m, n$, we set
$\tB(n) = B(R^n)$ and we set $\tdA(m,n) = A(R^m,R^n)$. 
For $n \geq 1$ we then set $k_n = |\tB(n-1)|^{-1} \chi_{\tB(n-1)}$ and
$h_n = \chi_{\tdA(n,n+1)}$, and $g_n = h_n*k_n$. 
\end{notation}

These definitions imply that $h_n^* = h_n$,
and the support of $g_n$ is contained in $A(R^n-R^{n-1}, R^{n+1} + R^{n-1})$.
We now fix a 
parameter $R$ of
the form $R=2^K$, with $K\in\naturals$ to be chosen later.
In particular, $R\ge 2$.
This $R$ will be used implicitly for much of the rest of this section. Then from the inequality \eqref{generaldoubling} we obtain
\[
|\tB(n-1)|^{-1} |\tB(n+1)| \leq \Clth^{2K}.
\]
Notice that the bound on the right is independent of $n$. 
Notice also that
$R^{n+1} - R^{n-1} \geq 2R^{n-1}$ because $R \geq 2$.

In the series of results below we employ the following notation. By
$C_k$ we denote a finite, positive quantity which depends only on the constant
$\Clth$ in the formulation \eqref{polygrowthdefn} of the bounded doubling hypothesis for $\Lth$,
and on the supplementary quantity $R$ which is to be chosen later in the proof.
In particular, each $C_k$ is independent of quantities $n,N$ that appear in the analysis.
Explicit expressions for each of these constants as functions of $\Clth,R$ can be extracted from
the steps below, but their precise values are of no intrinsic significance for our purposes.

We can apply Corollary \ref{power} to obtain:

\begin{lem}
\label{nilin}
For any $f \in c_c$ and for any $n \geq 1$ we have
\[
\|\l_{g_n f}\| \le C_1 R^{-n} J_D(f)
\]
where $C_1 =  \Clth^K$.
\end{lem}

It is natural to ask whether there exist length functions without bounded doubling
for which this lemma has an analogue.

\begin{prop}
\label{disjoint}
If $|n-m| \geq 2$ then $g_n$ and $g_m$ have disjoint support. 
\end{prop}

\begin{proof}
We can assume that $n > m$. If $g_m(x) \neq 0$ then $\Lth(x) \leq R^{m+1} + R^{m-1}$,
while if $g_n(x) \neq 0$ then $R^n - R^{n-1} < \Lth(x)$. 
But $R^{m+1} + R^{m-1} < R^n - R^{n-1}$ because $R \geq 2$ and $n-m \geq 2$.
\end{proof}

In particular,
$g_{2n}$ and $g_{2(n+1)}$ have disjoint support. Because
of this, we for the moment restrict to using these functions.
From Lemma~\ref{nilin} and $R\geq 2$ we obtain,
for any integer $N \geq 1$,
\[ \|\sum_{n \geq N}\l_{g_{2n} f}\|  \leq \sum_{n \geq N} R^{-2n} C_1 J_D(f)
= 2 C_1 R^{-2N} J_D(f) .  \]
Accordingly:
\begin{notation}
\label{psub}
Set $p_N = p_N^f = \sum_{n \geq N} g_{2n} f$.
\end{notation}
We then have:  

\begin{prop}
\label{double}
For any integer $N \geq 1$ 
\[ \|\l_{p_N}\| \leq 2C_1 R^{-2N} J_D(f)   .  \]
\end{prop}

Although Proposition~\ref{lip} gives some information about $L_D(g_nf)$,
we have not seen how to get a useful bound for $L_D(p_N)$. In contrast,
by using the support properties of the $g_n$'s we can obtain the following 
useful bound for $J_D(p_N)$, that is independent of $N$:

\begin{prop} \label{jip2}
For any positive integer $N$, \[ J_D(p_N^f) \le C_2 J_D(f) \]
where $C_2=4RC_1$.
\end{prop}

\begin{proof}
Fix $N$, and let $r>0$ be given. Let $N_r$ be the
biggest $M$ such that for $n<M$ the 
annulus $A(R^{2n}-R^{2n-1}, R^{2n+1} + R^{2n-1})$ is contained in $B(r)$,
that is, such that $R^{2n+1} + R^{2n-1} \leq r$.
If $\xi \in c_c$ has its support in $B(r)$, then for any $n < N_r$ the support of
$\l_{g_{2n}f}\xi$ is contained in $B(2r)$, and so $(I-M_{2r})\l_{g_{2n}f}\xi = 0$.
Thus, for $n < N_r$ 
\[ (I-M_{2r})\l_{g_{2n}f}M_r = 0.  \]
Consequently, by Lemma~\ref{nilin} 
\begin{align*}
 \|(I-M_{2r})\l_{p_N}M_r\| & \leq \sum_{n \geq N_r} \|\l_{g_{2n}f}\|   \\
&\leq \sum_{n \geq N_r} C_1 R^{-2n} J_D(f)  =  2 C_1 R^{-2N_r} J_D(f)   .
\end{align*}
Now from the definition of $N_r$ we have
\[
r \leq R^{2N_r+1} + R^{2N_r-1} \leq 2R^{2N_r+1}
\]
because $R \geq 2$. Thus $R^{2N_r} \geq r/(2R)$. On using this
in the previous displayed equation, we obtain:
\[
 \|(I-M_{2r})\l_{p_N}M_r\| \leq 2(2R/r) C_1 J_D(f). 
\]
Since this is true for all $r>0$, the proof is complete.
\end{proof}

Now set $q_N = q_N^f = f - p_N$. Notice that $q_N(x) = 0$ 
when for some $n \geq N$ we have $g_{2n}(x) = 1$, 
which from Lemma~\ref{support} happens when
\[
R^{2n} + R^{2n-1} < \Lth(x) \leq R^{2n+1} - R^{2n-1}   .
\]
Thus $q_N$ is supported in the union of the annular regions $A_n = A(s_n,t_n)$, with 
\[
s_n=R^{2(n-1)+1} - R^{2(n-1)-1} \text{ and } t_n = R^{2n} + R^{2n-1}.
\]

We now arrange to apply Proposition \ref{keyprop} to
control $\lambda_{f\chi_{A(s_n,t_n)}}$.
We seek $r_n$ such that $3r_n < s_n=R^{2n-3}(R^2-1)$. 
To ensure that $q_N^f$ vanishes on $A(s_n-2r_n,s_n)$
it suffices to have $s_n-2r_n \geq R^{2(n-1)} + R^{2(n-1)-1}$, that is,
\[ 
2r_n < R^{2n-1}-R^{2n-2} - 2R^{2n-3} = R^{2n-3}(R^2 - R -2), 
\]
while its vanishing on $A(t_n,t_n+2r_n)$ is ensured if
$t_n+2r_n \leq R^{2n+1} - R^{2n-1}$, that is, if
\[
 2r_n <  R^{2n+1} - R^{2n} - 2R^{2n-1} = R^{2n-1} (R^2 - R - 2). 
 \]
Assuming henceforth that $R\ge 4$,
it is easily checked that $r_n = \tfrac16 R^{2n-1}$ satisfies all
three of these conditions. 

We can now apply Proposition~\ref{keyprop}.
With the values of $r_n,s_n,t_n$ chosen above,
\[ A_n = A(s_n,t_n)= A(R^{2(n-1)+1} - R^{2(n-1)-1}, R^{2n} + R^{2n-1}).     \]
Then by inequality \eqref{generaldoubling},
\[
 |B(r_n)|^{-1}|B(t_n + r_n)| \le \Clth^{1+\log_2((t_n+r_n)/r_n)} = \Clth^{1+\log_2(6R+7)}. 
  \]
The uniform (with respect to $n$) boundedness of these ratios is crucial to our analysis
and relies on the bounded doubling hypothesis.
This uniform boundedness, in combination with Proposition~\ref{keyprop}, gives
\begin{equation} \label{replacesplit}
\|\l_{(q_N} \chi_{A_n)}\| \leq 
C_3 R^{-2n} J_D(q_N)
\end{equation}
where $C_3$ depends only on $\Clth,R$.

From Proposition~\ref{jip2} we obtain
\[
 J_D(q_N^f) \leq J_D(f) + J_D(p_N^f) \leq (1 + C_2)J_D(f), 
 \]
which together with inequality \eqref{replacesplit} establishes
\begin{lem}
\label{qcut}
With notation as above, for each $n$ 
\[ 
\|\l_{(q_N} \chi_{A_n)}\| \leq C_4 R^{-2n}  J_D(f) 
 \]
where $C_4 = (1+C_2)C_3$.
\end{lem}

\begin{notation}
\label{esub}
Set $\rho_N = \rho_N^f  = \sum_{n \geq N} q_N^f\chi_{A_n}$.
\end{notation}
Notice that if $\Lth(x) > R^{2N} + R^{2N-1}$ 
then $\rho_N^f(x) = q_N^f(x)$, so that $f - (p_N + \rho_N)$
is supported in $B( R^{2N} + R^{2N-1})$.
Much as in the
proof of Proposition~\ref{double} we obtain from the last displayed inequality above:

\begin{prop}
\label{sum}
With notation as above,  for any integer $N \geq 2$, 
\[
\|\l_{\rho_N} \| \leq 2C_4 R^{-2N} J_D(f).
\]
\end{prop}

But we also need control of $J_D(\rho_N)$:

\begin{prop}
\label{econt}
With notation as above, for any integer $N \geq 2$ 
\[
J_D(\rho_N^f) \leq 4C_4 J_D(f).
\]
\end{prop}

\begin{proof}  
The proof is very similar to that of Proposition~\ref{jip2}, but we
give the details since the bookkeeping is somewhat different.
Fix $N$, and let $r>0$ be given. Let $N_r$ be the
biggest $M$ such that for $n<M$ the 
annulus $A_n$ is contained in $B(r)$,
that is, such that $R^{2n} + R^{2n-1} \leq r$.
If $\xi \in c_c$ has its support in $B(r)$, then for any $n < N_r$ the support of
$\l_{(q_N\chi_{A_n})}\xi$ is contained in $B(2r)$, and so $(I-M_{2r})\l_{(q_N\chi_{A_n})}\xi = 0$.
Thus, for $n < N_r$ we have
\[ (I-M_{2r})\l_{(q_N\chi_{A_n})}M_r = 0.  \]
Consequently, by Lemma~\ref{qcut} we have
\begin{align*}
 \|(I-M_{2r})\l_{(q_N\chi_{A_n})}M_r\| & \leq \sum_{n \geq N_r} \|\l_{(q_N\chi_{A_n})}\|   \\
&\leq \sum_{n \geq N_r}R^{-2n} C_4 J_D(f) = 2C_4 R^{-2N_r} J_D(f)   .
\end{align*}
Now from the definition of $N_r$ we have
\[
r \leq R^{2N_r} + R^{2N_r-1} \leq 2R^{2N_r}
\]
because $R \geq 4$. Thus $R^{2N_r} \geq r/2$. On using this
in the previous displayed equation, we obtain:
\[
 \|(I-M_{2r})\l_{(q_N\chi_{A_n})}M_r\| 
\le 4C_4 r^{-1} J_D(f).
\]
Since this is true for all $r>0$, this concludes the proof.
\end{proof}

Proposition~\ref{prosaic}, and its extension concerning arbitrary functions
for which $[D_\Lth,\lambda_f]$ is bounded, have now been established.

\medskip
We finally assemble the pieces to conclude the proof of our main theorem.
Let $\e > 0$ be given. We will show that the set
\[
B_J = \{\l_f: f \in c_c, \ f(e) = 0, \ \ \mathrm{and} \ \ J_D(f) \leq 1\}
\]
can be covered by a finite number of $\e$-balls for the operator norm.
Since $J_D \leq L_D$, this will imply the same result for $L_D$ 
in place of $J_D$ above, which verifies the criterion of Proposition
\ref{proplip}, and so proves the assertion of our main
theorem. Note that up to this point we have not shown that
$B_J$ is bounded for the operator norm.

Fix $R \geq 4$, and choose $N \geq 2$ such that 
\[
R^{-2N}\max(C_1,C_4)<\e/4.
\]
From Propositions~\ref{double}  and \ref{sum}
it now follows that if $f \in B_J$ then 
\[
\max\big(\|\l_{p_N^f}\|, \|\l_{\rho _N^f}\|\big) < \e/4
\]
so that
\[ \| \l_{p_N^f+\rho_N^f}\| < \e/2  .  \]
Thus
\[ \|\l_f - \l_{f - (p_N + \rho_N)}\| < \e/2   .  \]

We need next to know that the set of functions of the form $f - (p_N^f+\rho_N^f)$ with $f \in B_J$
is bounded for the operator norm. To do this we first show that it is
bounded for the norm $J_D$.
From Propositions~\ref{jip2} and \ref{econt} it follows 
that for any $f \in B_J$ we have
\[
J_D(f - (p_N^f+\rho_N^f)) \leq J_D(f) +  C_2 J_D(f)   + 4C_4 J_D(f) 
\leq 1+C_2+4C_4,
\]
giving the desired boundedness for $J_D$.

Now by construction $f - (p_N^f+\rho_N^f)$ is 
supported in $B( R^{2N} + R^{2N-1})$.
Let 
\[
V_J^N = \{f \in c_c: \ f(e) = 0, \ \ \mathrm{and}  \ f \ 
\mathrm{is \ supported \ in} \ B( R^{2N} + R^{2N-1}) \}  .
\]
Let 
\[ B_J^N = \{f \in V_J^N: J_D(f) \leq 1+ C_2+4C_4 \}   , \]
and notice that each $f - (p_N^f+\rho_N^f)$ is in $B_J^N$.
Both $J_D$ and the operator norm (via $\l$) restrict to norms on the vector space $V_J^N$,
and these norms are equivalent because $V_J^N$ is finite-dimensional.
Thus $B_J^N$ is bounded for the operator norm. Since we have shown
above that every $f \in B_J$ is in the operator-norm $\e/2$-neighborhood
of an element of $B_J^N$, it follows that $B_J$ is bounded for the
operator norm.

Since $V_J^N$ is finite-dimensional, 
$B_J^N$ can be covered by a finite number of operator-norm $\e/2$-balls. 
Consequently, since $B_J$ is contained in the operator-norm
$\e/2$-neighborhood of $B_J^N$, it follows that $B_J$ can be covered by a
finite number of operator-norm $\e$-balls. Thus $B_J$ is totally bounded for the
operator norm. This concludes the proof of Theorem~\ref{mainth}.
\qed

\section{On polynomial growth}
\label{examp}

Proposition~\ref{bnded} states that strong polynomial growth implies the 
bounded doubling property,
which implies polynomial growth, and that these are equivalent for finitely generated groups.

\begin{proof}[Proof of Proposition~\ref{bnded}]
Suppose that $\Lth$ has strong polynomial growth. Then, with
notation as in Definition \ref{polygr}, for any strictly
positive $r, s$ we get
\[
 |B(s)| \leq c^2 s^dr^{-d} |B(r)|   ,
 \]
which for $s = 2r$ gives the bounded doubling property.
Suppose instead that $\Lth$ has bounded doubling.
Then for any $s\ge 1$ we get
$|B(2^ks)| \leq \Clth^k |B(s)|$  
for each nonnegative integer $k$.
From this we find that if $1 \le s \le r$, then
\begin{equation}
\label{generaldoubling} 
|B(r)| \leq \Clth^{1+\log_2(r/s)} |B(s)|  
\end{equation}
where $\log_2$ denotes the base $2$ logarithm.
Indeed, let $k$ be the positive integer that satisfies $2^{k-1}s<r\le 2^ks$. 
Then $|B(r)|\le |B(2^ks)| \le \Clth^k|B(s)|$ and $k-1\le \log_2(r/s)$. On
setting $s = 1$ and rearranging we see that $\Lth$ has polynomial growth.  

Suppose now that $G$ is finitely generated
and that $\Lth$ is a length function on $G$. Then
for any word-length function $\tilde \Lth$ on $G$ there 
exists $C<\infty$ such that $\Lth\le C^{-1} \tilde \Lth$,
that is, the balls $\tilde B(r)$ associated to $\tilde \Lth$ satisfy $\tilde B(r)\subset B(Cr)$. 
Thus if $\Lth$ has polynomial growth, it follows that $\tilde \Lth$ does also.
 According to a theorem of Gromov \cite{Grm1, Klr,  Mnn, ShT}, 
this implies that $G$ is nilpotent-by-finite. 
But the property of strong polynomial growth holds for any word-length function on 
a finitely generated nilpotent-by-finite group \cite{Wlf, Bss, Mnn}. Thus
$\tilde\Lth$ has strong polynomial growth, and so there are
constants $\tilde C_{\tilde \Lth}$ and $\tilde d$ such that
\[
\tilde C_{\tilde \Lth}^{-1}r^{\tilde  d} \leq |\tilde B(r)| \leq |B(Cr)| 
\]
for all $r\ge 0$. This implies that $\Lth$ has strong polynomial growth.
\end{proof}

We conclude by exhibiting simple examples illustrating the inequivalence between these growth
properties, for groups that are not finitely generated.
Chapter~9 of \cite{Mnn} also contains an interesting discussion of infinitely generated groups
that are of locally polynomial growth.

\begin{exam}
The function $\Lth(x)=\ln(2|x|)$ for all $x\ne 0$
on the group $G=\integers$ is a length function that is not of polynomial growth.
\end{exam}

The remaining examples are based on infinite direct sums of finite groups.
Let $(G_n)_{n\in\naturals}$ be an arbitrary sequence of finite groups, with identity elements $e_n$.
Let $G$ be the direct sum of all these groups; $G$ consists of all sequences $x=(x_1,x_2,x_3,\dots)$
with $x_n\in G_n$ for all $n$ and $x_n = e_n$  for all but finitely many indices $n$.
Multiplication is defined componentwise.
Let $e = (e_1,e_2,\dots)$ be the identity element of $G$.
Let $1\le a_1<a_2<a_3<\dots$ be a strictly increasing sequence of positive real numbers
satisfying $\lim_{n\to\infty} a_n=\infty$.
Define $\Lth:G\to[0,\infty)$ by $\Lth(e)=0$
and $\Lth(x) = \max_{n: x_n\ne e_n} a_n$ for all $x\ne e$. 
Then $\Lth$ is a proper length function. 
Moreover, if $r=a_n$ then $|B(r)| = \prod_{m=1}^n |G_m|$.

\begin{exam}
Let $G_n= \integers/2\integers$, the group with $2$ elements.
Let $a_k = 2^{k^2}$. 
Then $|B(2^{K^2})| = 2^K$ for all $K\in\naturals$ and more  generally
$|B(r)| \le e^{C\sqrt{\ln(r)}}$ for all $r\ge 2$, for a certain constant $C<\infty$.
Thus the growth rate of $\Lth$ is slower than polynomial, and so $\Lth$
can not have strong polynomial growth. But if $r \geq 2$ and if the natural
number $p$ is such that $2^{p^2} \leq r < 2^{(p+1)^2}$ so that 
$|B(r)| = 2^p$, then $2r \leq 2^{(p+1)^2}$ so that $|B(2r)| \leq 2^{p+1}$.
Thus $|B(2r)| \leq 2|B(r)|$, so that $\Lth$ has bounded doubling.
\end{exam}

\begin{exam}
Now choose $(G_n)$ so that $|G_n|>1$ for all $n$ 
and $\lim_{n\to\infty} |G_n|=\infty$.
Choose $a_n = \prod_{m=1}^n |G_m|$.
The balls on the product group $G$ satisfy
$|B(a_n)| = \prod_{m=1}^n |G_m| = a_n$ for all $n$, 
and $|B(r)|<r$ for all other $r>1$, so $\Lth$ has polynomial growth. 
However, for $2\le r=a_n$, $\frac{|B(r)|}{|B(r/2)|} \ge \frac{|B(r)|}{|B(r-1)|} = |G_n|$ is not bounded above
uniformly in $n$, and so the doubling property does not hold.
\end{exam}

The next example shows that $\Lth$ can have polynomial growth, yet grow irregularly.
\begin{exam}
Let $G$ be as above. Choose any two parameters
$1< \gamma_1<\gamma_2<\infty$, 
and let $1=N_1<N_2<N_3<\dots$ be a sequence tending to infinity.
Set $a_1=1$ and for $N_k\le n<N_{k+1}$
choose $a_{k+1}/a_k=\gamma_1$ if $k$ is odd, and $=\gamma_2$ if $k$ is even.
Then $\Lth$ has polynomial growth.
However, $\Lth$ need not have strong polynomial growth.
Indeed, it is plainly possible to arrange, 
by choosing the sequence $(N_k)$ to increase to infinity sufficiently rapidly, that 
\[ \limsup_{r\to\infty} \frac{\log|B(r)|}{\log r} = \gamma_1^{-1}
\ \text{ while } \ 
\liminf_{r\to\infty} \frac{\log|B(r)|}{\log r} = \gamma_2^{-1}.\]
\end{exam}

\begin{exam}
Let $G^0$ be a finite non-commutative simple group,
and let $\gamma>1$.
Choose $G_n=G^0$ for all $n$, and $a_n = \gamma^n$.
Then $\Lth$ has polynomial growth, yet $G$ is not nilpotent-by-finite.
\end{exam}


\begin{thebibliography}{10}

\bibitem{AnC}
Cristina Antonescu and Erik Christensen.
\newblock Metrics on group {$C\sp *$}-algebras and a non-commutative
  {A}rzel\`a-{A}scoli theorem.
\newblock {\em J. Funct. Anal.}, 214(2):247--259, 2004.
\newblock arXiv:math.OA/0211312.

\bibitem{Bss}
Hyman Bass.
\newblock The degree of polynomial growth of finitely generated nilpotent
  groups.
\newblock {\em Proc. London Math. Soc. (3)}, 25:603--614, 1972.

\bibitem{Crs2}
Michael Christ.
\newblock Inversion in some algebras of singular integral operators.
\newblock {\em Rev. Mat. Iberoamericana}, 4(2):219--225, 1988.

\bibitem{Crs1}
Michael Christ.
\newblock On the regularity of inverses of singular integral operators.
\newblock {\em Duke Math. J.}, 57(2):459--484, 1988.

\bibitem{Cn1}
Alain Connes.
\newblock {$C\sp{\ast} $} alg\`ebres et g\'eom\'etrie diff\'erentielle.
\newblock {\em C. R. Acad. Sci. Paris S\'er. A-B}, 290(13):A599--A604, 1980.

\bibitem{Cn7}
Alain Connes.
\newblock Compact metric spaces, {F}redholm modules, and hyperfiniteness.
\newblock {\em Ergodic Theory Dynamical Systems}, 9(2):207--220, 1989.

\bibitem{dHp}
Pierre de~la Harpe.
\newblock Groupes hyperboliques, alg\`ebres d'op\'erateurs et un th\'eor\`eme
  de {J}olissaint.
\newblock {\em C. R. Acad. Sci. Paris S\'er. I Math.}, 307(14):771--774, 1988.

\bibitem{Grm1}
Mikhael Gromov.
\newblock Groups of polynomial growth and expanding maps.
\newblock {\em Inst. Hautes \'Etudes Sci. Publ. Math.}, (53):53--73, 1981.

\bibitem{Jo2}
Paul Jolissaint.
\newblock Rapidly decreasing functions in reduced ${C}\sp *$-algebras of
  groups.
\newblock {\em Trans. Amer. Math. Soc.}, 317(1):167--196, 1990.

\bibitem{Klr}
Bruce Kleiner.
\newblock A new proof of {G}romov's theorem on groups of polynomial growth.
\newblock {\em J. Amer. Math. Soc.}, 23(3):815--829, 2010.

\bibitem{Mnn}
Avinoam Mann.
\newblock {\em How groups grow}, volume 395 of {\em London Mathematical Society
  Lecture Note Series}.
\newblock Cambridge University Press, Cambridge, 2012.

\bibitem{OzR}
Narutaka Ozawa and Marc~A. Rieffel.
\newblock Hyperbolic group {$C\sp *$}-algebras and free-product
  {$C\sp*$}-algebras as compact quantum metric spaces.
\newblock {\em Canad. J. Math.}, 57(5):1056--1079, 2005.
\newblock arXiv:math.OA/0302310.

\bibitem{pck1}
Judith~A. Packer.
\newblock {$C^*$}-algebras generated by projective representations of the
  discrete {H}eisenberg group.
\newblock {\em J. Operator Theory}, 18(1):41--66, 1987.

\bibitem{pck2}
Judith~A. Packer.
\newblock Twisted group {$C^*$}-algebras corresponding to nilpotent discrete
  groups.
\newblock {\em Math. Scand.}, 64(1):109--122, 1989.

\bibitem{Ptn}
Alan L.~T. Paterson.
\newblock {\em Amenability}, volume~29 of {\em Mathematical Surveys and
  Monographs}.
\newblock American Mathematical Society, Providence, RI, 1988.

\bibitem{R30}
Marc~A. Rieffel.
\newblock Proper actions of groups on {$C^*$}-algebras.
\newblock In {\em Mappings of operator algebras ({P}hiladelphia, {PA}, 1988)},
  volume~84 of {\em Progr. Math.}, pages 141--182. Birkh\"auser Boston, Boston,
  MA, 1990.

\bibitem{R4}
Marc~A. Rieffel.
\newblock Metrics on states from actions of compact groups.
\newblock {\em Doc. Math.}, 3:215--229, 1998.
\newblock arXiv:math.OA/9807084.

\bibitem{R5}
Marc~A. Rieffel.
\newblock Metrics on state spaces.
\newblock {\em Doc. Math.}, 4:559--600, 1999.
\newblock arXiv:math.OA/9906151.

\bibitem{R18}
Marc~A. Rieffel.
\newblock Group {$C\sp *$}-algebras as compact quantum metric spaces.
\newblock {\em Doc. Math.}, 7:605--651, 2002.
\newblock arXiv:math.OA/0205195.

\bibitem{R6}
Marc~A. Rieffel.
\newblock Gromov-{H}ausdorff distance for quantum metric spaces.
\newblock {\em Mem. Amer. Math. Soc.}, 168(796):1--65, 2004.
\newblock arXiv:math.OA/0011063.

\bibitem{R31}
Marc~A. Rieffel.
\newblock Integrable and proper actions on {$C^*$}-algebras, and
  square-integrable representations of groups.
\newblock {\em Expo. Math.}, 22(1):1--53, 2004.

\bibitem{R7}
Marc~A. Rieffel.
\newblock Matrix algebras converge to the sphere for quantum
  {G}romov-{H}ausdorff distance.
\newblock {\em Mem. Amer. Math. Soc.}, 168(796):67--91, 2004.
\newblock arXiv:math.OA/0108005.

\bibitem{R21}
Marc~A. Rieffel.
\newblock Leibniz seminorms for ``{M}atrix algebras converge to the sphere''.
\newblock In {\em Quanta of Maths}, volume~11 of {\em Clay Mathematics
  Proceedings}, pages 543--578, Providence, R.I., 2011. Amer. Math. Soc.
\newblock arXiv:0707.3229.

\bibitem{R29}
Marc~A. Rieffel.
\newblock Matricial bridges for ``matrix algebras converge to the sphere''.
\newblock 2015.
\newblock arXiv:1502.00329.

\bibitem{ShT}
Yehuda Shalom and Terence Tao.
\newblock A finitary version of {G}romov's polynomial growth theorem.
\newblock {\em Geom. Funct. Anal.}, 20(6):1502--1547, 2010.

\bibitem{Wlf}
Joseph~A. Wolf.
\newblock Growth of finitely generated solvable groups and curvature of
  {R}iemanniann manifolds.
\newblock {\em J. Differential Geometry}, 2:421--446, 1968.

\end{thebibliography}


\def\dbar{\leavevmode\hbox to 0pt{\hskip.2ex \accent"16\hss}d}

\end{document}